\documentclass[letterpaper, 10 pt, conference]{ieeeconf}
\IEEEoverridecommandlockouts 

\usepackage{amsmath,amssymb, ntheorem}
\usepackage{psfrag}
\usepackage{cite}
\usepackage{lipsum}
\usepackage{graphicx}
\usepackage{algorithm}
\usepackage{algorithmicx}
\usepackage{algpseudocode}
\usepackage[font=small,labelfont=bf]{caption}
\usepackage{graphics}
\usepackage{epstopdf}
\usepackage{verbatim}
\usepackage{float}
\usepackage{dsfont}
\usepackage{framed}
\usepackage{textcomp}
\usepackage{subcaption}
\usepackage{enumerate}
\usepackage{accents}
\restylefloat{figure}
\usepackage{color,xspace}
\usepackage{tikz}
\usepackage{upgreek}
\usetikzlibrary{automata,positioning,fit,shapes.geometric}
\usepackage{accents}

\usepackage{url}

\newcommand{\goesto}{\rightarrow}

\newcommand{\real}{\mathbb{R}}

\newcommand{\NP}{\mathcal{NP}}
\newcommand{\bE}{\mathbb{E}}

\usepackage[normalem]{ulem}
\makeatletter
\renewtheoremstyle{plain}
{\item{\theorem@headerfont ##1\ ##2\theorem@separator}~}
{\item{\theorem@headerfont ##1\ ##2\ (##3)\theorem@separator}~}
\makeatother

{\theoremheaderfont{\upshape\bfseries}
	\theorembodyfont{\normalfont\slshape}	
	\newtheorem{defn}{Definition}
	\newtheorem{lem}{Lemma}
	
	\newtheorem{theorem}{Theorem}

	\newtheorem{remark}{\textbf{Remark}}{}
}
\renewenvironment{proof}{\noindent\textbf{Proof:}}{}

\floatstyle{plain}
\restylefloat{figure}

\newcommand{\oprocendsymbol}{\hbox{$\bullet$}}
\newcommand{\oprocend}{\relax\ifmmode\else\unskip\hfill\fi\oprocendsymbol}

\pgfdeclarelayer{bg}
\pgfdeclarelayer{bbg}
\pgfsetlayers{bbg,bg,main}

\newcommand{\Infect}{\mathbb{I}}
\newcommand{\Cost}{\mathcal{C}}

\newcommand{\Graph}{\mathcal{G}}
\newcommand{\Nodes}{\mathcal{V}}
\newcommand{\Edges}{\mathcal{E}}

\newcommand{\HealTime}{\mathcal{H}} 
\newcommand{\SusNodes}{\mathcal{S}}
\newcommand{\InfectNodes}{\mathcal{I}}
\newcommand{\Neighborsi}{\mathcal{N}_i}
\newcommand{\ProNodes}{\mathcal{P}}

\newcommand{\tune}{\mu}

\newcommand{\Labels}{\mathcal{L}}

\newcommand{\Action}{\mathcal{A}}

\begin{document}
\title{\huge Control of Generalized Discrete-time SIS Epidemics via\\ Submodular Function Minimization}
\author{%
	Nicholas J. Watkins and George J. Pappas%
	\thanks{N.J. Watkins and G.J. Pappas are with the Department of Electrical and Systems Engineering, University of Pennsylvania, Philadelphia, PA 19104, USA, {\tt\small \{nwatk,pappasg\}@seas.upenn.edu}. }}
\maketitle
\begin{abstract}
	In this paper, we study a novel control method for a generalized $SIS$ epidemic process.  In particular, we use predictive control to design optimal protective resource distribution strategies which balance the need to eliminate the epidemic quickly against the need to limit the rate at which protective resources are used.  We expect that such a controller may be useful in mitigating the spread of biological diseases which do not confer immunity to those who have been infected previously, with sexually transmitted infections being a prominent example of such.  Technically, this paper provides a novel contribution in demonstrating that the particular combinatorial optimal control problem used to design resource allocations has an objective function which is submodular, and so can be solved in polynomial time despite its combinatorial nature.  We test the performance of the proposed controller with numerical simulations, and provide some comments on directions for future work.
\end{abstract}

\section{Introduction} \label{sec:intro}

Due to increased population mobility and a rapidly changing climate, the likelihood of a major biological epidemic occuring in modern times is significant \cite{Arias2009}.  It is therefore prudent to invest effort into understanding the control of epidemics.  Epidemic control is not a new field of research.  Indeed, the modern study of epidemic processes dates back to at least the 1920s\cite{Kermack1927}, with work focusing specifically on the optimal control of spreading processes dating back to at least the 1970s \cite{Wickwire1975,Sethi1978,Behncke2000,Elena2013}.  Notably, nearly all such work focuses on the analysis and control of mean-field approximations to epidemic processes (see, e.g., \cite{Nowzari2016} for a recent review).  While known to be accurate for homogeneous populations asymptotically (i.e. in the limit of large populations) \cite{Gast2012}, it is unclear how well mean-field approximations work in practical settings, as they ignore the underlying stochastic nature of the contacts which drive disease dissemination.  This has given rise to recent work studying the control of epidemics without using mean-field approximation.  

It was shown in \cite{Drakopoulos2014} that varying the healing rate of an $SIS$ process sufficiently aggressively enables a controller to drive the contagion out of the network quickly.  It was shown in \cite{Scaman2016} that for any priority order strategy, there exists a sufficiently large budget so as to guarantee that the epidemic is driven out of the network quickly.  It was shown in \cite{Watkins2017c} that it is possible to construct a system of ordinary differential equations which gives provable upper- and lower- bounds on infection statistics, and to use these approximations to enforce a stability constraint that guarantees fast extinction of the epidemic.  However, in all of these works, the optimization problems which fundamentally characterize the controllers are $\NP$-hard.  As such, only suboptimal solutions may be obtained, where the performance attained may vary.

It is interesting to consider if there is any way to design predictive controllers for stochastic spreading process models in such a way that the underlying optimization is tractable.  Indeed, it has been known for some time that optimizing seed selection under general threshold spreading models is a submodular maximization problem \cite{Kempe2003, Mossel2007}, and so can be well approximated by a greedy algorithm.  Similar sorts of structural properties have been identified in many variations of problems on threshold-type spreading models, (see, e.g., \cite{He2016, Kuhlman2013,Bogunovic2012}), but such models are only appropriate for modeling non-recurrent epidemics.  That is, they study processes in which each agent in the network can only actively spread the phenomenon for one interval of time, and afterwards remains inactive.  While such a feature seems appropriate in many contexts, e.g. when a disease confers immunity to survivors, this is not always the case.  Indeed, many sexually transmitted infections (e.g. chlamydia, gonorrhea) do not confer such immunity (see, e.g., \cite{Kretzschmar1996}).

The study of discrete-time models with recurrent compartmental memberships seems to have originated in \cite{Gomez2010}, where the stability of a deterministic approximation to a discrete-time Susceptible-Infected-Susceptible (SIS) model was studied.  Since then, deterministic approximations to discrete-time models have been studied in a few contexts \cite{Ahn2014,Ruhi2015,Ruhi2016,Eksin2017}.  While it is known that in some cases the open-loop stability of the deterministic approximation implies a similar notion of open-loop stability of the stochastic model \cite{Ahn2014,Ruhi2015,Ruhi2016}, it appears that \cite{Watkins2017} was the first to consider closed-loop control of such processes.  Whereas \cite{Watkins2017} studied the control of the discrete-time $SIS$ process by direct control of the process parameters, here we study the control of a generalized $SIS$ model by way of direct allocation of protective resources.

\paragraph*{\textbf{Statement of Contributions}}
\emph{To the best of the authors' knowledge, this paper is the first to develop a feedback controller for discrete-time, recurrent epidemic processes using node removal as a means of actuation.}  With respect to prior work, the present article builds on a preliminary paper \cite{Watkins2017}, which studied the control of a discrete-time $SIS$ process by direct control of the process parameters.  Principally, our work here differs in that we focus on control via node removal.  This change in focus allows us to model control actions more realistically: protective devices (e.g. latex gloves, barriers, condoms) are discrete objects by nature, and so should be treated as such in the model.  

The primary technical contribution of this work is in demonstrating that a particular combinatorial optimal control problem can be posed as a submodular minimization problem with respect to a ground set that is a subset of the set of nodes, and so can be solved efficiently (i.e., in polynomial time with respect to the size of the graph).  Additionally, we investigate the performance of the controller by way of numerical simulations, and provide some comments on directions for future work which we believe to be promising.

\section{Problem Statement} \label{sec:problem}

In this section, we formally develop the problem studied in the paper.  Section \ref{subsec:epidemic} details the generalized $SIS$ model studied.  Section \ref{subsec:actuation} details the actuation model studied.  Section \ref{subsec:opt_con} details the optimal control problem studied.

\subsection{Epidemic Model} \label{subsec:epidemic}
We study an epidemic process on an $n$ node graph $\Graph,$ with node set $\Nodes,$ and edge set $\Edges.$  Epidemic processes are mathematical objects which evolve on graphs, in which each node is considered to be an agent, and each edge a relationship between agents.  The current health status of an agent is modeled by a collection of compartments, where at each time, every agent belongs to exactly one compartment.  The process we study has one susceptible (i.e. healthy) compartment (denoted by $S$), and $p$ infected spreading compartments (denoted by $\{I_k\}_{k \in [p]}$).  A transition from susceptibility to infection comes from contact between a susceptible node and an infected node.  Other transitions are due to random events internal to the agent (e.g. progressing through the stages of the disease).  Note that the multiplicity of infected compartments here is used to shape the amount of time an agent spends in infection.

We denote by $X_i^{\ell}(t)$ an indicator random variable which takes the value $1$ if node $i$ is in compartment $\ell$ at time $t,$ and takes the value $0$ otherwise.  We take compartmental memberships to evolve in a way that generalizes the standard discrete-time $SIS$ process \cite{Gomez2010}.  A node $i$ transitions from susceptible to the first infected compartment (i.e. $I_1$) on the increment from time $t$ to time $t+1$ through contact with nodes which are in an infected compartment at time $t.$  Given that a node $i$ is infected at time $t,$ it transitions to other model compartments independently of all external phenomena.
We then have that the indicator $X_i^{S}$ evolves as
\begin{equation}
	\label{eq:S}
	X_i^{S+} = X_i^{S}(1 - Z_{i(X)}^{S \goesto I_1}) + \sum_{k = 1}^{p} X_i^{I_k} Z_{i}^{I_k \goesto S},
\end{equation} 
where the random variable $Z_{i(X)}^{S \goesto I_1}$ is an indicator that at least one infection event (indicated by random variables $Y_{ij}$) has occured on an edge between node $i$ and an infected node $j$ in the set of neighbors of $i$ (denoted $\Neighborsi$), and is defined as
\begin{equation}
\label{eq:z_def}
Z_{i(X)}^{S \goesto I_1} \triangleq \min \{1, \sum_{j \in \Neighborsi} Y_{ij} \sum_{k = 1}^{p} X_j^{I_k} \},
\end{equation}
and the random variables $Z_{i}^{I_k \goesto S}$ are indicators denoting a transition from the $k$'th infected compartment to susceptibility.  Likewise, we have that the indicator $X_i^{I_1}$ evolves as  
\begin{equation}
\label{eq:I1}
X_i^{I_{1}+} = X_i^{S}Z_{i(X)}^{S \goesto I_1} + \sum_{k = 1}^{p} X_i^{I_k} Z_{i}^{I_k \goesto I_1} - X_i^{I_1} \sum_{\ell \in \Labels} Z_{i(X)}^{I_1 \goesto \ell},
\end{equation}
and the indicators $X_i^{I_k}$ for $k \neq 1$ evolve as
\begin{equation}
\label{eq:Ik}
X_i^{I_{k}+} = \sum_{\ell \in \Labels \setminus \{S\}} X_i^{\ell} Z_{i}^{\ell \goesto I_{k}} - X_i^{I_k} \sum_{\ell \in \Labels} Z_{i}^{I_k \goesto \ell},
\end{equation}
where $\Labels$ is the set of all compartmental labels.
Note that for all nodes to belong to a unique compartment at all times, we must have that $Z_{i(X)}^{\ell \goesto \ell^\prime}$ each take values on $\{0,1\}$ and satisfy $\sum_{\ell^\prime} Z_{i(X)}^{\ell \goesto \ell^\prime} = 1$ with probability one.  

To make matters concrete, we fix one particular way of enforcing this.  We assume all $Y_{ij}$ in \eqref{eq:z_def} are independently distributed Bernoulli random variables with known success probabilities, and that when a node is infected, the compartmental membership random variables evolve independently as a discrete-time Markov chain with $p+1$ states, structured so that $S$ is its unique absorbing state, and that $S$ is reached in finite time with probability one.  This level of generality allows us to model the amount of time taken to recover from an infection with more precision than a standard $SIS$ epidemic, in which there is only one infected compartment, and $Z_{i}^{I_1 \goesto S}$ is an independent Bernoulli random variable at all times.  Such an assumption forces the time taken to recover to be distributed as a geometric random variable.  Under the model specified here, the time taken to recover follows a discrete phase-type distribution, and so is quite general \cite{Bobbio2003}.

\subsection{Actuation Model} \label{subsec:actuation}

We consider allocating protective barriers in order to preventing the spread of an infection.  Formally, our controller actuates the process detailed in Section \ref{subsec:epidemic} by selecting a subset of nodes $\ProNodes$ to protect against infection.  Because protective devices (e.g. latex gowns, gloves, condoms) are often single-use, it is appropriate to model the economic cost of protecting the set of nodes $\ProNodes$ as being the sum over all edges in which one adjacent node is a susceptible node that is protected, and the other adjacent node is infected.  

That is, we only pay to protect a particular node if it interacts with an infected person, and we pay in proportion to the extent of interaction between the protected node and its infected neighbors over the discretized time period.  For example, if we are allocating condoms to mitigate the spread of chlamydia and we want to protect a particular person $i,$ we will provide $i$ with one condom for each contact between it and all infected partners.  Mathematically, we have
\begin{equation} 
	\label{eq:cost_def}
	\Cost(\ProNodes|X) \triangleq \sum_{i \in \SusNodes(X)} \mathds{1}_{\{i \in \ProNodes \}} \sum_{j \in \Neighborsi \cap \InfectNodes(X)} c_{ij},
\end{equation}
where $\mathds{1}_{\{ \centerdot \}}$ is a $\{0,1\}$ indicator function, the non-negative constants $c_{ij}$ model the cost of providing a protective barrier for all interactions between $i$ and $j$ over one time period, $\SusNodes(X)$ denotes the set of susceptible nodes at state $X,$ and $\InfectNodes(X)$ denotes the set of infected nodes at state $X.$  We model protecting a node by removing it from the graph.  That is, a node $i$ which is in the set of protected nodes $\ProNodes$ and is susceptible at time $t$ is susceptible at time $t+1$ with probability one.  Mathematically, we then have that the controlled dynamics for $X_i^{S}$ follow  
\begin{equation}
\label{eq:Scontrol}
X_i^{S+} = X_i^{S}(1 - \mathds{1}_{\{i \notin \ProNodes\}}Z_{i(X)}^{S \goesto I_1}) + \sum_{k = 1}^{p} X_i^{I_k} Z_{i}^{I_k \goesto S},
\end{equation} 
and the controlled dynamics for $X_i^{I_1}$ are changed similarly.

\subsection{Optimal Control Problem} \label{subsec:opt_con}

If we were strictly concerned with minimizing the total accumulated cost for our controller, it is feasible to take the action $\ProNodes = \emptyset$ at all times.  Doing so incurs zero cost.  However, this may result in the infection persisting in the population for a long time.  This is a socially undesirable outcome.  In general, we may wish to consider the presence of infection as a sort of soft cost imposed on the controller.

Previous works studying optimal control of epidemics (e.g., \cite{Wickwire1975,Sethi1978,Behncke2000,Elena2013}) often do so by posing a problem of the form
\begin{subequations}
	\begin{equation}
		\label{prog:inf_hor}
		\min_{\pi \in \Pi} \bE_{\pi}  [\sum_{\tau = 0}^{\infty} J(X(\tau), \theta(\tau) ) | X ],
	\end{equation}
	\begin{equation}
		\label{eq:inf_hor_cost}
		J(X,\theta) \triangleq \tune c(\theta) + (1-\tune) q(\Infect(X))
	\end{equation}
\end{subequations}
where $\Pi$ is the set of all non-anticipating control policies which map observations of the process as it evolves to changes in the spreading model's parameters $\theta,$ $c(\theta)$ is the instantaneous cost of setting the processes parameters to $\theta,$ $\Infect(X)$ is a function which maps the state $X$ to the number of infected nodes in state $X,$ and the function $q$ is some non-negative valued function which determines the extent to which we should care about the existence of infection.  In our case, if it were so that the process $\{X(t)\}$ was a Markov chain on a small state space and the set of possible control actions were small, \eqref{prog:inf_hor} could be solved by treating it as a Markov decision process and applying a standard solution technique, e.g. value iteration.  However, even for standard $SIS$ processes, $\{X(t)\}$ evolves as a Markov chain with $2^n$ states, and there are $2^n$ possible choices of the set $\ProNodes.$  That is, the complexity of using a standard Markov devision process algorithm here is exponential, due to the curse of dimensionality.  As such, it is an interesting task to construct a controller which allows for the same qualitative tradeoff as \eqref{prog:inf_hor}, but is computationally tractable to implement.

To accomplish this, we consider applying controls which solve the infinite-horizon optimal control problem
\begin{subequations}
	\begin{equation} 
	\label{prog:mpc}
	\min_{\ProNodes \subseteq \Nodes} h_{X}(\ProNodes)
	\end{equation}
	\begin{equation}
	\label{eq:weighted_cost}
	h_{X}(\ProNodes) \triangleq \mu \Cost(\ProNodes|X) + (1-\mu) Q(\ProNodes|X)
	\end{equation}
\end{subequations}
where $\Cost(\ProNodes|X)$ is the cost function \eqref{eq:cost_def}, and $Q(\ProNodes |X)$ is the cost of applying the rollout policy which protects all susceptible nodes for all future times, given that the set of nodes $\ProNodes$ is protected at the current time (see, e.g., \cite[Section 6.4]{BertsekasDynV1} for background on the use of rollout policies in infinite horizon optimal control).  Intuitively, our controller designs the protection set $\ProNodes$ at the current time while \emph{anticipating} that at all future times, every reasonable action will be implemented to eradicate infection as quickly as possible.  Thus, our controller anticipates that a cost $c_{ij}$ occurs at every time where exactly one of the node $i$ or $j$ is susceptible.  That is, $Q$ is defined mathematically as
\begin{equation}
	\label{eq:Qdef}
	\begin{aligned}
	&Q(\ProNodes|X) \triangleq \\
	&\hspace{40 pt}\bE_{\Theta(\ProNodes)}[\sum_{\tau = 1}^{\infty} \sum_{\{i,j\} \in \Edges} (X_i^{S}(\tau) \oplus X_j^{S}(\tau)) c_{ij} | X],
	\end{aligned}
\end{equation}
	where $\Theta(\ProNodes)$ is the measure induced by protecting the set of nodes $\ProNodes$ at the current time, and by $\oplus$ we denote the exclusive or operator, i.e. for two variables $Y,Z \in \{0,1\},$
\begin{equation} \label{eq:xor}
	Y \oplus Z \triangleq 
	\begin{cases}
	0, & Y = 0, Z = 0;\\
	1, & Y = 1, Z = 0;\\
	1, & Y = 0, Z = 1;\\
	0, & Y = 1, Z = 1.
	\end{cases}
\end{equation}  

Note that because we update our decision $\ProNodes$ at every time, this total protection strategy is not actually implemented.  Rather, $Q$ plays the same role in \eqref{prog:mpc} as $q(\Infect)$ plays in \eqref{prog:inf_hor}.  It penalizes the existence of infection in the network, and so allows a control designer a means for trading off between an immediate resource expenditure and the rate of decay in the population of infected individuals by appropriately selecting $\tune.$  However, it is not immediately clear that actions can be efficiently computed as in \eqref{prog:mpc}, as the set over which the optimum must be computed has $2^n$ elements, and it is not clear that the objective $h_{X}$ is sufficiently well structured so as to enable efficient optimization.

\emph{In the body of this paper, we are concerned with determining whether solutions of \eqref{prog:mpc} can be computed efficiently (i.e. in polynomial time with respect to the size of the graph), and whether the engendered controller provides reasonable behavior.}  We address computational efficiency in Section \ref{subsec:computing}, in which we show that the objective of the minimization in \eqref{prog:mpc} is submodular, which allows us to minimize $h_{X}$ using polynomial time algorithms.  We assess the controller's behavior in Section \ref{sec:example} with a numerical example.

\section{Predictive Control of Generalized $SIS$} \label{sec:optimization}

In this section, we develop the mathematical foundations required to efficiently implement a controller as described in Section \ref{sec:problem}.  Section \ref{subsec:math} provides some technical preliminaries which are needed in our analysis.  Section \ref{subsec:computing} demonstrates that the optimization problem \eqref{prog:mpc} has sufficient structure so as to allow for the use of efficient optimization algorithms, and provides some references to software packages that can be used to solve problem.

\subsection{Mathematical Preliminaries} \label{subsec:math}
The key mathematical concept which allows for the efficient solution of \eqref{prog:mpc} is submodularity.  Submodularity is a mathematical formalization of the concept of diminishing returns: adding an object to a larger set has less of an impact than adding the same object to a smaller set. Formally, a submodular function satisfies the following definition.

\begin{defn}[Submodular Functions] \label{defn:submod}
	Let $\Omega$ be a finite ground set of objects, and suppose $f:2^\Omega \goesto \real,$ where $2^{\Omega}$ denotes the power set of $\Omega,$ i.e. the set of all subsets of $\Omega.$  The function $f$ is said to be \emph{submodular} if and only if
	\begin{equation} \label{ineq:submod}
		f(Z \cup \{z\}) - f(Z) \leq f(Y \cup \{z\}) - f(Y)
	\end{equation}
	holds for all $Y \subset Z \subseteq \Omega,$ and $z \in \Omega \setminus Z.$ \oprocend
\end{defn} 

It is frequently the case that the submodularity of a complicated function is verified by reducing the proof to checking the submodularity of a simpler function.  We use such an argument later (Section \ref{subsec:computing}), using the exclusive-or function as the simple function, which is submodular:
\begin{lem}[Submodularity of Restricted Exclusive Or] \label{lem:xor}
	Let $\Omega$ be a finite ground set, take $a, b \in \Omega,$ $A,B \in \{0,1\},$ and let $\oplus$ denote the operator defined by \eqref{eq:xor}.  The function
	\begin{equation*}
		f(W) = A \mathds{1}_{\{a \notin W\}} \oplus B \mathds{1}_{\{b \notin W\}},
	\end{equation*}
	is submodular.
\end{lem}
\begin{proof}
	Note that if $A =B = 0,$ then $f(W) = 0,$ and is trivially submodular.  Note also that if $A \neq B,$ then $f(W) = A \mathds{1}_{\{a \notin W\}} + B \mathds{1}_{\{b \notin W\}},$ which is non-negative sum of submodular functions, so is submodular~\cite{McCormick2005a}.  It remains to consider $A = B = 1.$  In this case, if $a = b,$ we have $f(W) = 0,$ which is trivially submodular.  It remains to consider the case in which $a \neq b.$
	
	Take $Z \supseteq Y,$ $z \in \Omega \setminus Z,$ and suppose neither $a$ nor $b$ are in $Z.$  By $Z \supseteq Y,$ we have that neither $a$ nor $b$ are in $Y.$  If $z \in \{a,b\},$ then $f(Z \cup \{z\}) - f(Z) = 1 = f(Y \cup \{z\}) - f(Y).$  If $z \notin \{a,b\},$ then $f(Z \cup \{z\}) - f(Z) = 0 = f(Y \cup \{z\}) - f(Y).$
	
	Now, take $Z \supseteq Y,$ $z \in \Omega \setminus Z,$ and suppose both $a$ and $b$ are in $Z.$  The condition $z \in \Omega \setminus Z$ implies that $z \notin \{a,b\},$ and so $f(Z \cup \{z\}) - f(Z) = 0 = f(Y \cup \{z\}) - f(Y).$ 
	
	Finally, take $Z \supseteq Y,$ $z \in \Omega \setminus Z,$ and suppose exactly one of $a$ or $b$ is in $Z.$  For concreteness, suppose $a \in Z.$  There are two cases to consider.  First, suppose $a \notin Y.$  Then, if $z \notin \{b\},$ we have $f(Z \cup \{z\}) - f(Z) = 0 = f(Y \cup \{z\}) - f(Y).$  If $z \in \{b\},$ then $f(Z \cup \{z\}) - f(Z) = -1 \leq f(Y \cup \{z\}) - f(Y) = 1.$  Now, suppose $a \in Y.$  If $z \notin \{b\},$ then $f(Z \cup \{z\}) - f(Z) = 0 = f(Y \cup \{z\}) - f(Y).$  If $z \in \{b\},$ then $f(Z \cup \{z\}) - f(Z) = -1 = f(Y \cup \{z\}) - f(Y).$  \oprocend
\end{proof}

\subsection{Efficiently Computing Optimal Sets of Protected Nodes} \label{subsec:computing}

A principal reason that submodularity is an important concept is that it allows for a variety of combinatorial optimization problems to be solved (or approximately solved) in polynomial time.  In this subsection, we demonstrate that \eqref{prog:mpc} is one such problem.  As it is well-known that the minimum of a submodular function over a finite ground set can be computed in time which grows polynomially with respect to the size of the ground set (see, e.g., \cite{McCormick2006,Iwata2007}), it suffices for us to demonstrate that the objective of \eqref{prog:mpc} is a submodular function.  We accomplish this in the following result:

\begin{theorem}[Submodularity of Objective Function]
	\label{thm:submodular}
	Fix a particular generalized $SIS$ process (as detailed in Section \ref{subsec:epidemic}), a number $\tune \in [0,1],$ and a state $X.$  Let $\Cost$ be defined as in \eqref{eq:cost_def}, and $Q$ as in \eqref{eq:Qdef}.  Then, the function
	$h_{X}(\ProNodes)$ as defined by \eqref{eq:weighted_cost} is a submodular set function, where the ground set is taken to be the set of susceptible nodes at state $X,$ $\SusNodes(X).$
\end{theorem}

\begin{proof}
	Since a function defined by a non-negative weighted sum of submodular functions is itself a submodular function~\cite{McCormick2005a}, it suffices to show that $\Cost$ and $Q$ are both submodular functions.  We handle each separately.
	
	\paragraph*{Submodularity of $\Cost$}
	Because the constants $c_{ij}$ are non-negative, it suffices to note that indicators of the form $f(W) = \mathds{1}_{\{r \in W\}}$ are submodular.  This is straightforward: if $x \in \Omega \setminus W,$ then $f(W \cup \{x\}) - f(W) = 1$ when $x = r,$ and $0$ otherwise, regardless of the choice of $W.$  This implies that \eqref{ineq:submod} is met with equality, satisfying Definition \ref{defn:submod}.
	
	\paragraph*{Submodularity of $Q$}
	Since the generalized $SIS$ process defined in Section \ref{subsec:epidemic} is defined on a finite state space in finite time, we may construct a finite sample space for it.  Call this sample space $\Xi.$  Because having exactly one node in a pair of nodes $i$ and $j$ be susceptible is equivalent to having exactly one node in a pair of nodes belong to some infected compartment, we have
	\begin{equation*}
		\begin{aligned}
		&Q(\ProNodes | X) = \bE_{\Theta(\ProNodes)}[\sum_{\tau = 1}^{\infty} \sum_{\{i,j\} \in \Edges} (X_i^{I}(\tau) \oplus X_j^{I}(\tau)) c_{ij} | X],\\
		&= \sum_{\xi \in \Xi} \sum_{\tau = 1}^{\infty} \sum_{\{i,j\} \in \Edges} (X_i^{I}(\tau; \ProNodes, \xi) \oplus X_j^{I}(\tau; \ProNodes,\xi)) c_{ij} \Pr(\xi | X),
		\end{aligned}
	\end{equation*}
	where we denote the dependence on the choice of protected nodes and the underlying sample $\xi$ explicitly, use the notation $\Pr(\xi|X)$ for the probability that the sample $\xi$ is drawn from $\Xi,$ given that the process is currently in state $X,$ and define the shorthand notation $X_j^{I} = \sum_{k = 1}^{p} X_i^{I_k}.$  Since the terms $c_{ij}$ and $\Pr(\xi | X)$ are nonnegative, it suffices to show that the terms $(X_i^{I}(\tau; \ProNodes, \xi) \oplus X_j^{I}(\tau; \ProNodes,\xi))$ are submodular functions.
	
	Because the rollout policy removes all susceptible nodes which are adjacent to an infected node, we have that all nodes which are infected at some time $\tau \geq t+1$ were either infected at time $t,$ or became infected on the transition from time $t$ to time $t+1.$  In the case where the node $i$ was infected at time $t,$ the definition of the process has that all indicators $X_i^{\ell}(\tau; \ProNodes,\xi)$ are independent of the choice $\ProNodes,$ and are determined only by the particular choice of $\xi.$  In particular, we have that
	\begin{equation*}
	\begin{aligned}
	&X_i^{I}(\tau; \ProNodes, \xi) = \\
	&\hspace{10 pt} \begin{cases}
		\mathds{1}_{\{\tau \leq \HealTime_{i}(\xi,X_i(t))\}}, & X_i^{I} = 1;\\
		\mathds{1}_{\{i \notin \ProNodes\}} \mathds{1}_{\{ \lor_{j \in \Neighborsi} Y_{ij}(\xi) X_j^{I} \}} \mathds{1}_{\{\tau \leq \HealTime_{i}(\xi,X_i(t+1) \}}, & X_i^{S} = 1,
	\end{cases}
	\end{aligned}
	\end{equation*}
	where $\HealTime_{i}(\xi,X_i(t))$ is a random variable denoting the next time at which an event which transitions the node $i$ from infected to susceptible happens, given that the node is in its current compartment,
	and $\lor$ denotes the logical or operation, used here because all infection events have the same effect on an unprotected susceptible node.
	
	Note that in the case that $X_i^I = 1,$ submodularity follows immediately from noting that $X_i^{I}(\tau; \ProNodes, \xi)$ is unaffected by the choice $\ProNodes.$  In the case that $X_i^{S} = 1,$ we have $X_i^{I}(\tau; \ProNodes, \xi)$ is the indicator $\mathds{1}_{\{i \notin \ProNodes\}}$ multiplied by a constant in $\{0,1\},$ which is determined by $\xi.$  It follows that $Q(\ProNodes|X)$ is a non-negative weighted sum of functions of the form
	\begin{equation*}
		f(\ProNodes) = A \mathds{1}_{\{a \notin \ProNodes\}} \oplus B \mathds{1}_{\{b \notin \ProNodes\}},
	\end{equation*}
	with $A,B \in \{0,1\},$ which are submodular by Lemma \ref{lem:xor}.
	\oprocend
\end{proof}

The submodularity proven in Theorem~\ref{thm:submodular} allows us to use a variety of polynomial-time algorithms \cite{Krause2010a,McCormick2006} to solve \eqref{prog:mpc} to optimality.  In our simulations (Section \ref{sec:example}), we have used an implementation of Fujishige's minimum-norm-point algorithm \cite{Fujishige2006,Chakrabarty2014a} from the submodular function optimization package \cite{Krause2010a}.  Alternatively, one may consider using faster \emph{approximate} submodular minimization algorithms \cite{Jegelka2011,Chakrabarty2016} if the problem instance considered is too large to be handled efficiently in practice by other methods.  It is also worth noting that \eqref{prog:mpc} may possess sufficient structure to enable the development of faster algorithms for computing the exact optimum, as it is the sum of a monotone increasing modular function $\Cost(\ProNodes|X),$ and a submodular function $Q(\ProNodes|X)$ which is a weighted sum of indicators and exclusive-or operations.

\begin{figure}
	\centering
	\includegraphics[width=0.5\textwidth]{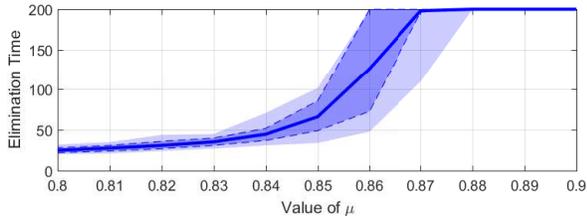}
	\caption{The effect of varying $\tune$ to tune controller.  As expected, the controller's behavior experiences a phase transition from a domain in which future cost is prioritized (here, below $\mu = 0.8$), and a domain in which instantaneous spending is prioritized (here, above $\mu =0.88$).  Note: the simulation parameters reported in the example of Section \ref{sec:example} were used to generate this figure.} \label{fig:tuning}
\end{figure}

\begin{remark}[On State and Parameter Uncertainty] \label{rem:imperfect}
	{\rm
	It may be the case in practice that neither the state $X$ nor the parameters of the spreading process are deterministically known.  However, this situation can be readily accommodated into our controller's computations.  Supposing that one has knowledge of the distribution of the state and spreading parameters, one can first sample from this, and then forward-simulate the evolution of the process with the sampled state and parameter set fixed.  Verifying the submodularity of the objective function with respect to each sample path used can be done with the same argument as in the proof of Theorem \ref{thm:submodular}, and so merits no further discussion here.  \oprocend
	}
\end{remark}

\begin{remark}[Tuning Controller Behavior] \label{rem:tuning}
	{\rm
		Selecting $\tune$ appropriately is important.  For small values of $\tune,$ the objective of \eqref{prog:mpc} is dominated by the effect of $Q(\Action | X).$  In the extreme case where $\tune = 0,$ the controller takes an action which minimizes the future cost associated to guaranteeing the infection spreads no further.  For large values of $\tune,$ the objective of \eqref{prog:mpc} is dominated by the effect of $\Cost(\Action |X).$  In the extreme case where $\tune = 1,$ it is simple to show that the optimal action will be to protect no nodes, under any circumstance.  Intuitively, intermediate values of $\tune$ balance the immediate cost of resource expenditure against its long-term consequences.  This is seen in Figure \ref{fig:tuning}.
		\oprocend}
\end{remark}

\section{A Numerical Example} \label{sec:example}

We consider a graph with $200$ nodes.  We assume that for each $i,j \in \Nodes$ such that $i\neq j,$ the edge $\{i,j\}$ is in the set of edges $\Edges$ independently with probability $0.01.$  We assume that an unprotected contact between two nodes results in an infection spreading from the infected node to the susceptible node with a probability randomly generated from the unit interval.  We assume that nodes come into contact with each other a maximum of three times per day, and so edge costs take values in the set $\{0,1,2,3\},$ where the value $0$ is taken if the edge $\{i,j\} \notin \Edges,$ and the other values are chosen uniformly at random if $\{i,j\} \in \Edges.$  Because treatments for sexually transmitted infections (e.g. chlamydia) often take in excess of a week to be effective \cite{CDC2018}, the Markov chain used to model the infectious compartments was chosen so as to make the distribution of times from infection to susceptibility uniformly distributed on the set $\{7,8,9,10\}.$

A plot demonstrating the effect of varying $\tune$ is given in Figure \ref{fig:tuning}, where the lightest shaded region contains the middle $98\%$ of the $100$ sample paths generated, the darker shade of blue contains the middle $80\%$ of the sample paths generated, and the central blue line gives the sample average.  As anticipated in Remark \ref{rem:tuning}, there is a continuous range of values $\tune$ over which the controller's behavior markedly changes.  For this particular example, this region was found to be the interval $[.8, .9].$  Note that in general, appropriate values of $\tune$ must be determined for each fixed model.  It may well be possible to automate a procedure for finding appropriate values of $\tune$ online.  However, developing and applying any such method would likely be quite technically involved, and is best saved as a task for future work.
	
The behavior of the controller with $\tune = 0.85$ is given in Figure \ref{fig:sim}, where the lightest shaded region contains the middle $98 \%$ of the 100 generated samples, the darker shaded region contains the middle $80 \%$ of the generated samples, and the central blue line gives the sample average of the process.  We see that the controller drives the epidemic out of the network quickly, while using resources at a lesser rate than one which protects all susceptible nodes at all times.
\begin{figure}
	\centering
	\begin{subfigure}{0.5\textwidth}
		\centering
		\includegraphics[width=\textwidth]{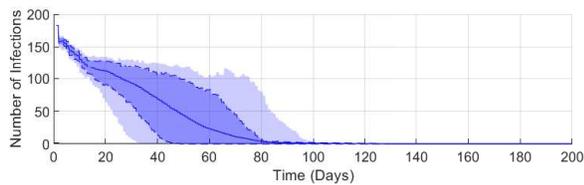}
		\caption{The number of infected nodes as a function of time.} \label{subfig:infection}
	\end{subfigure}
	
	\begin{subfigure}{0.5\textwidth}
		\centering
		\includegraphics[width=\textwidth]{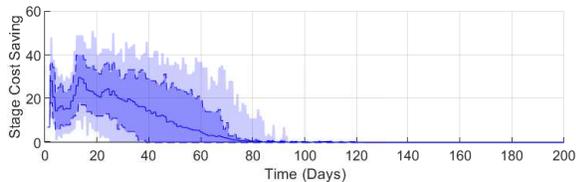}
		\caption{Stage cost savings compared against protecting all nodes.} \label{subfig:stagecost}
	\end{subfigure}
	\caption{A study of the controller performance at $\tune = 0.85.$  Figure \eqref{subfig:infection} demonstrates that the controller eliminates the epidemic from the network quickly, despite its aggressive decisions.  Figure \eqref{subfig:stagecost} shows that the controller chooses actions with considerable cost savings when compared against protecting all susceptible nodes.} \label{fig:sim}
\end{figure}

\section{Conclusions and Future Work} \label{sec:conclude}
In this paper, we have developed a controller which allocates discrete protective devices in order to mitigate the spread of an epidemic process.  Such a controller could be of use in fighting many forms of biological disease, which often do not confer immunity to people that have survived infections, and so are well-modeled by $SIS$-type processes.

As topics for future work, one may consider the task of developing efficient observers and estimators for the process' state and parameters.  While Remark \ref{rem:imperfect} suggests how the controller can be used in the case where the state and spreading parameters are not perfectly known, developing numerically efficient methods for providing the required distributions remains an important open task.  We believe the work presented here provides clear motivation for future researchers to engage in such work, and so is valuable itself.

\bibliography{library}
\bibliographystyle{ieeetr}

\end{document}